\documentclass[12pt,reqno]{amsart}
\usepackage{faktor}

\usepackage[english]{babel}
\usepackage[utf8]{inputenc}

\usepackage{faktor}
\usepackage{ragged2e}
\usepackage{enumitem}
\usepackage{graphicx}
\usepackage{subfig}
\usepackage{tikz-cd}
\usetikzlibrary{arrows.meta}

\usepackage[T1]{fontenc} 
\usepackage{amsmath,amsthm,amsfonts,amssymb,dsfont,mathtools} 
\usepackage{multirow,multicol}
\usepackage{tabularx} 
\usepackage{setspace} 
\usepackage{verbatim}

\usepackage{color}
\usepackage[dvipsnames]{xcolor}

\usepackage[pdftex,%
bookmarks=true,
breaklinks=true,
bookmarksnumbered = true,
colorlinks= true,
urlcolor= green,
anchorcolor = yellow,
citecolor=blue,
]{hyperref}                   

\newtheoremstyle{theoremm}{}{}{\itshape}{}{\scshape}{.}{ }{}
\theoremstyle{theoremm}
\newtheorem{teo}{Theorem}[section]
\newtheorem{lema}[teo]{Lemma}
\newtheorem{prop}[teo]{Proposition}
\newtheorem{coro}[teo]{Corollary}

\newtheoremstyle{remark}{}{}{}{}{\scshape}{.}{ }{}
\theoremstyle{remark}
\newtheorem{defi}[teo]{Definition}

\renewcommand{\ker}[1]{\ensuremath{\operatorname{\text{Ker}}\left({#1}\right)}}

\renewcommand{\to}{\ensuremath{\longrightarrow}}
\renewcommand{\mapsto}{\ensuremath{\longmapsto}}

\makeatletter
\def\@map#1#2[#3]{\mbox{$#1 \colon\thinspace #2 \to #3$}}
\def\map#1#2{\@ifnextchar [{\@map{#1}{#2}}{\@map{#1}{#2}[#2]}}
\makeatother

\numberwithin{equation}{section}
\allowdisplaybreaks

\usepackage{fancyhdr}
\everymath{\displaystyle}

\begin{document}

\title{A crystallographic quotient of the framed braid group and associated flat manifolds}

\author[{\^E}nio Leite]{{\^E}nio Leite}
\address{{\^E}nio Leite\\ Secretaria da Educa\c c\~ao do Estado da Bahia, 5ª Avenida, nº 550, no Centro Administrativo da Bahia (CAB),  CEP 41745-004 - Salvador - BA - Brazil.}
\email{leiteenio@gmail.com}

\author[Oscar Ocampo]{Oscar Ocampo}
\address{Oscar Ocampo\\ Universidade Federal da Bahia, Departamento de Matem\'atica - IME, CEP:~40170-110 - Salvador, Brazil}
\email{oscaro@ufba.br}

\author[P.~C.~C.~Santos J\'unior]{Paulo Cesar Cerqueira dos Santos J\'unior}
\address{Paulo Cesar Cerqueira dos Santos J\'unior\\ Universidade Estadual do Sudoeste da Bahia, Departamento de Ci\^encias Exatas e Tecnol\'ogicas, Estrada do Bem Querer, Km 04, CEP:~45083-900 - Vit\'oria da Conquista - BA - Brazil}
\email{paulo.santos@uesb.edu.br}

\author[A.~C.~M.~Teixeira]{Ana Carolina Moura Teixeira} 
\address{Ana Carolina Teixeira\\ Universidade Federal da Bahia, Departamento de Matem\'atica - IME, CEP:~40170-110 - Salvador, Brazil}\email{ana.teixeira@ifba.edu.br}


\date{\today}

\begin{abstract}
We study the quotient of the framed braid group $FB_n$ by the commutator subgroup of the framed pure braid group $FP_n$. We prove that
$$
\faktor{FB_n}{[FP_n,FP_n]}\cong\mathbb Z^n\rtimes\faktor{B_n}{[P_n,P_n]},
$$
and, for $n\geq3$, that this quotient is a crystallographic group of dimension $n(n+1)/2$ with holonomy group $S_n$. We investigate its finite-order elements and conjugacy classes, prove that it has no elements of finite even order, and construct Bieberbach subgroups with prescribed finite $2$-group holonomy. Starting from suitable Bieberbach subgroups of $\faktor{B_n}{[P_n,P_n]}$, we then construct Bieberbach subgroups of the framed quotient whose associated compact flat manifolds are torus bundles over the corresponding unframed flat manifolds. We describe these bundles through the generalized Calabi construction and determine their orientability, first Betti numbers, and the centers of their fundamental groups. As an application, we study a family with cyclic holonomy of odd order, obtaining a criterion for the existence of flat K\"ahler structures, and indicate how the same construction applies to cyclic holonomy of $2$-power order.
\end{abstract}

\maketitle

\section{Introduction}\label{sec:intro}

Braid groups constitute an important class of groups arising naturally in topology, knot theory, configuration spaces, and geometric group theory. The Artin braid group $B_n$ is naturally identified with the fundamental group of the unordered configuration space of $n$ distinct points in the plane, while the pure braid group $P_n$ corresponds to the ordered configuration space. This relationship between braid groups and configuration spaces goes back to the work of Fadell and Neuwirth \cite{FaN}, who established, among other fundamental results, the exact sequences associated with fundamental groups of configuration spaces. Since then, braid groups have played a central role in several areas of topology and group theory.

A natural extension of the Artin braid group is obtained by allowing each strand to carry an integer framing. The resulting group is the \emph{framed braid group} $FB_n$. Algebraically, it may be described as the semidirect product
$$
FB_n=\mathbb Z^n\rtimes B_n,
$$
where $B_n$ acts on $\mathbb Z^n$ by permuting the framing indices. Framed braid groups were introduced in the context of $3$-manifold theory by Ko and Smolinsky \cite{KS}, who used them in the formulation of an analogue of the Markov theorem for representations of $3$-manifolds. Thus, framed braid groups enrich classical braid theory by incorporating both crossing and framing data.

A natural source of crystallographic groups in braid theory comes from quotients associated with (pure) braid groups. Gonçalves, Guaschi, and Ocampo \cite{GGO1} proved that, for $n\geq3$, the quotient $\faktor{B_n}{[P_n,P_n]}$ is a crystallographic group with holonomy group $S_n$, and studied its torsion, finite subgroups, and conjugacy classes. This connection was subsequently extended to generalized braid groups associated with complex reflection groups \cite{Marin}, braid groups of orientable closed surfaces of positive genus \cite{GGOP}, and braid groups of punctured surfaces \cite{Renato}. Related crystallographic quotients were also obtained for virtual braid groups, virtual twin groups, and other braid-like groups in \cite{CO}. 

Recall that a crystallographic group of dimension $d$ is a discrete and cocompact group of isometries of $\mathbb R^d$. By the Bieberbach theorems, such a group contains a maximal normal free abelian subgroup of finite index, isomorphic to $\mathbb Z^d$, called its translation subgroup; the corresponding finite quotient is its holonomy group. A torsion-free crystallographic group is called a Bieberbach group, and Bieberbach groups are precisely the fundamental groups of compact flat Riemannian manifolds. We refer to \cite{Charlap,Szczepanski} for background on crystallographic groups and flat manifolds.

The first purpose of this paper is to study the analogue of the quotient $\faktor{B_n}{[P_n,P_n]}$ for framed braid groups. Let $FP_n$ denote the framed pure braid group. We prove that
$$
\faktor{FB_n}{[FP_n,FP_n]}\cong\mathbb Z^n\rtimes\faktor{B_n}{[P_n,P_n]},
$$
where the action on $\mathbb Z^n$ is induced by the natural permutation action on the framing indices. Moreover,
$$
\faktor{FP_n}{[FP_n,FP_n]}\cong\mathbb Z^n\times\faktor{P_n}{[P_n,P_n]}\cong\mathbb Z^{n(n+1)/2}.
$$
The action of $S_n$ on this free abelian group is the direct sum of the natural permutation action on the framing generators and the permutation action on the basis indexed by the two-element subsets of $\{1,\ldots,n\}$. Since this representation is faithful, for $n\geq3$ we obtain a short exact sequence
$$
1\longrightarrow\mathbb Z^{n(n+1)/2}\longrightarrow\faktor{FB_n}{[FP_n,FP_n]}\longrightarrow S_n\longrightarrow1,
$$
and hence $\faktor{FB_n}{[FP_n,FP_n]}$ is a crystallographic group of dimension $n(n+1)/2$ with holonomy group $S_n$.

We then investigate torsion and conjugacy in this crystallographic quotient. We prove that it has no elements of finite even order and that every element of odd order in $S_n$ admits a lift of the same order to $\faktor{FB_n}{[FP_n,FP_n]}$. In particular, the finite orders occurring in the framed quotient are precisely the odd orders occurring in $S_n$. We also apply the general conjugacy methods developed in \cite{CO} to obtain normal forms and conjugacy criteria for elements of the framed quotient. Finally, if $H\leq S_n$ is a $2$-subgroup, we prove that its preimage in $\faktor{FB_n}{[FP_n,FP_n]}$ is a Bieberbach group of dimension $n(n+1)/2$ with holonomy group $H$. Consequently, by Cayley's theorem, every finite $2$-group occurs as the holonomy group of a compact flat manifold whose fundamental group is isomorphic to a subgroup of the framed quotient.

We then use this structure to study flat manifolds obtained from Bieberbach subgroups of the classical braid quotient. Given a Bieberbach subgroup $\Gamma\leq\faktor{B_n}{[P_n,P_n]}$ of dimension $d$ with holonomy group $H\leq S_n$ and holonomy representation $\Theta$, its framed extension $\widetilde{\Gamma}=\mathbb Z^n\rtimes\Gamma$ is a Bieberbach subgroup of $\faktor{FB_n}{[FP_n,FP_n]}$ of dimension $n+d$, with the same holonomy group and holonomy representation $\rho\oplus\Theta$, where $\rho$ is the natural permutation representation. We prove that the associated flat manifold $M_{\widetilde{\Gamma}}$ is a geometric $T^n$-bundle over $M_\Gamma$ with totally geodesic fibers and an affine section, and we describe it via the generalized Calabi construction. If the action of $H$ on $\{1,\ldots,n\}$ has $s$ orbits, say $O_1,\ldots,O_s$, then the center of $\widetilde{\Gamma}$ is isomorphic to $Z(\Gamma)\times\mathbb Z^s$, while the first Betti number satisfies $b_1(M_{\widetilde{\Gamma}})=b_1(M_\Gamma)+s$. We also obtain an explicit orientability criterion in terms of the holonomy representation.

As a concrete application, we study the family of Bieberbach groups with cyclic holonomy of odd order constructed in~\cite{Ocampo}. If $q=p_1^{r_1}\cdots p_t^{r_t}$ and $n=p_1^{r_1}+\cdots+p_t^{r_t}$, the corresponding framed flat manifold is orientable, and both its first Betti number and the rank of the center of its fundamental group are equal to $(n+t^2)/2$. We also prove that it admits a flat K\"ahler structure if and only if $n+t^2\equiv0\pmod4$; in particular, for an odd prime power $q=p^r$, this is equivalent to $q\equiv3\pmod4$.

The same framed construction applies to the Bieberbach groups with cyclic holonomy of $2$-power order studied in~\cite{OR}. We record the resulting torus-bundle construction and determine the corresponding changes in the first Betti number and the center.

The paper is organized as follows. In Section~\ref{crys}, we recall the basic facts about crystallographic and Bieberbach groups used throughout the paper. We study the quotient $\faktor{FB_n}{[FP_n,FP_n]}$ in Section~\ref{fracry}; we first establish its semidirect product decomposition, then prove that it is crystallographic, investigate its torsion and conjugacy, and construct Bieberbach subgroups with prescribed finite $2$-group holonomy. In Section~\ref{sec:flat}, we construct framed Bieberbach groups from Bieberbach subgroups of $\faktor{B_n}{[P_n,P_n]}$, describe the associated flat manifolds as torus bundles and via the generalized Calabi construction, determine orientability, the center, and the first Betti number, study in detail the family with cyclic holonomy of odd order, and conclude by indicating how the same construction applies to the cyclic $2$-power families of \cite{OR}.

\subsection*{Acknowledgments}

E.~L.~was partially supported by Funda\c c\~ao de Amparo \`a Pesquisa do Estado da Bahia (FAPESB). O.~O.~was partially supported by the National Council for Scientific and Technological Development (CNPq, Brazil) through a \textit{Bolsa de Produtividade} grant No.~305422/2022--7.

\section{Crystallographic and Bieberbach groups}\label{crys}

In this section, we recall the basic facts about crystallographic and Bieberbach groups used throughout the paper, including an algebraic characterization of crystallographic groups in terms of extensions with finitely generated free abelian kernel and finite quotient, as well as some standard facts concerning Bieberbach groups and compact flat Riemannian manifolds. For further details on these topics, we refer the reader to \cite[Section I.1.1]{Charlap}, \cite[Section 2.1]{Dekimpe}, and \cite[Chapter 3]{Wolf}. Throughout this work, we identify $\operatorname{Aut}(\mathbb{Z}^n)$ with $\operatorname{GL}(n,\mathbb{Z})$.

A discrete and cocompact subgroup $\Pi$ of $\mathbb{R}^{n}\rtimes O(n,\mathbb{R})\subseteq\operatorname{Aff}(\mathbb{R}^{n})$ is called a crystallographic group of dimension $n$. If, in addition, $\Pi$ is torsion-free, then $\Pi$ is called a Bieberbach group of dimension $n$. By the Bieberbach theorems, every crystallographic group $\Pi$ of dimension $n$ contains a maximal normal free abelian subgroup of finite index, called its translation subgroup, which is isomorphic to $\mathbb{Z}^n$. The corresponding finite quotient is called the holonomy group of $\Pi$. Moreover, Bieberbach groups are precisely the fundamental groups of compact flat Riemannian manifolds.

Recall that an integral representation of rank $n$ of a group $\Phi$ is a homomorphism $\Theta\colon\Phi\longrightarrow\operatorname{Aut}(\mathbb{Z}^n)$. Two such representations are said to be equivalent if they are conjugate in $\operatorname{Aut}(\mathbb{Z}^n)$, and $\Theta$ is said to be faithful if it is injective. We shall use the following algebraic characterization of crystallographic groups; see \cite[Lemma 8]{GGO1} and \cite[Theorem 2.1.4]{Dekimpe}.

\begin{prop}\label{caraccrys}
A group $\Pi$ is crystallographic of dimension $n$ if and only if there exists a short exact sequence $1\longrightarrow\mathbb{Z}^{n}\longrightarrow\Pi\stackrel{\zeta}{\longrightarrow}\Phi\longrightarrow1$ such that $\Phi$ is finite and the integral representation $\Theta\colon\Phi\longrightarrow\operatorname{Aut}(\mathbb{Z}^n)$, induced by conjugation on $\mathbb{Z}^n$, is faithful. More precisely, for $\varphi\in\Phi$, choose $\pi\in\Pi$ such that $\zeta(\pi)=\varphi$ and define $\Theta(\varphi)(x)=\pi x\pi^{-1}$, for $x\in\mathbb{Z}^n$.
\end{prop}

Notice that the representation $\Theta$ is well defined. Indeed, if $\pi'$ is another lift of $\varphi$, then $\pi'=\pi a$ for some $a\in\mathbb{Z}^n$, and since $\mathbb{Z}^n$ is abelian, we have
$$
\pi'x\pi'^{-1}=\pi axa^{-1}\pi^{-1}=\pi x\pi^{-1}.
$$
In the situation of Proposition~\ref{caraccrys}, $\Phi$ is the holonomy group of $\Pi$, and $\Theta$ is called the (integral) holonomy representation of $\Pi$. Its extension of scalars $\Theta_{\mathbb R}\colon\Phi\longrightarrow GL(n,\mathbb R)$ is called the real holonomy representation. 

We shall also use the elementary fact that if $K$ is a characteristic subgroup of $N$ and $N$ is a normal subgroup of a group $G$, then $K$ is normal in $G$. In particular, for every normal subgroup $N\trianglelefteq G$, its commutator subgroup $[N,N]$ is normal in $G$.

The following corollary shows that subgroups of the holonomy group can be realized as holonomy groups of crystallographic subgroups of the same dimension.

\begin{coro}[{\cite[Corollary 10]{GGO1}}]\label{holocor}
Let $\Pi$ be a crystallographic group of dimension $n$ and holonomy group $\Phi$, and let $1\longrightarrow\mathbb{Z}^{n}\longrightarrow\Pi\stackrel{\zeta}{\longrightarrow}\Phi\longrightarrow 1$ be its associated short exact sequence. If $H\leq\Phi$, then $\zeta^{-1}(H)$ is a crystallographic subgroup of $\Pi$ of dimension $n$ with holonomy group $H$.
\end{coro}

We shall also use some terminology and results concerning geometric fibrations of crystallographic groups. Following Ratcliffe and Tschantz~\cite{RT1}, we use the term \emph{space group} for a crystallographic group. We write an isometry of $\mathbb R^m$ as $a+A$, with $a\in\mathbb R^m$ and $A\in O(m)$, acting by $x\mapsto a+Ax$. For a subgroup $N$ of a space group $\Pi$, set
$$
\operatorname{Span}_{\mathbb R}(N)=\operatorname{span}_{\mathbb R}\{a\in\mathbb R^m\mid a+I\in N\}.
$$
If $N$ is normal in $\Pi$ and $V=\operatorname{Span}_{\mathbb R}(N)$, then $N$ is called \emph{complete} if
$$
N=\{a+A\in\Pi\mid a\in V\text{ and }V^\perp\subseteq\operatorname{Fix}(A)\}.
$$
Here $V^\perp$ denotes the orthogonal complement of $V$ in $\mathbb R^m$, and $\operatorname{Fix}(A)=\{x\in\mathbb R^m\mid A(x)=x\}$. 

If $\faktor{\Pi}{N}$ is torsion-free, then $N$ is complete and the quotient map determines a geometric fiber bundle 
$$
V/N \longrightarrow \mathbb R^m/\Pi \longrightarrow (\mathbb R^m/V)/(\Pi/N)
$$
with totally geodesic fibers~\cite[Theorem~13]{RT1}. Here, a geometric fiber bundle means a fiber bundle which is locally isometrically equivalent to the projection of a product $F\times B$ onto the second factor, where $F$ is the fiber and $B$ is an open metric ball in the base. 

Moreover, if the extension 
$$
1 \longrightarrow N \longrightarrow \Pi \longrightarrow \faktor{\Pi}{N} \longrightarrow 1
$$ 
splits, then this fibration admits an affine section~\cite[Lemma~6]{RT1}. Here an affine section is a section of the bundle projection that lifts, with respect to the natural quotient maps, to an affine map between the corresponding Euclidean spaces.

\section{Framed braid groups and crystallographic groups}\label{fracry}

Let $B_n$ be the braid group on $n$ strands, generated by $\sigma_1,\sigma_2,\ldots,\sigma_{n-1}$ and subject to the relations
\begin{eqnarray}
\sigma_i\sigma_j&=&\sigma_j\sigma_i \qquad\mbox{for } |i-j|>1,\label{dis}\\
\sigma_i\sigma_j\sigma_i&=&\sigma_j\sigma_i\sigma_j \qquad\mbox{for } |i-j|=1,\label{pro}
\end{eqnarray}
where $i,j=1,\ldots,n-1$. We consider the natural action of $B_n$ on the set $\{1,2,\ldots,n\}$ induced by the canonical homomorphism $\pi\colon B_n\longrightarrow S_n$, where $\pi(\sigma_i)=(i,i+1)$. Throughout, we adopt the convention that permutations act on the right. Thus, if $\rho,\tau\in S_n$, then $(\rho\tau)(i)=\tau(\rho(i))$.

\begin{defi}[\cite{KS}]
The framed braid group $FB_n$ is the group generated by $\sigma_1,\ldots,\sigma_{n-1}$ and $t_1,\ldots,t_n$, subject to the relations~(\ref{dis}),~(\ref{pro}) and
$$
\begin{aligned}
t_it_j&=t_jt_i,\\
\sigma_kt_i&=t_{\sigma_k(i)}\sigma_k,
\end{aligned}
$$
where $i,j=1,\ldots,n$ and $k=1,\ldots,n-1$.
\end{defi}

Equivalently, the framed braid group can be described as the semidirect product
$$
FB_n=\mathbb{Z}^n\rtimes B_n,
$$
where $B_n$ acts on $\mathbb{Z}^n=\langle t_1,\ldots,t_n\rangle$ by permuting the generators according to the canonical homomorphism $\pi\colon B_n\longrightarrow S_n$. In particular, every element of $FB_n$ can be written uniquely in the form $t_1^{r_1}\cdots t_n^{r_n}\alpha$, where $r_1,\ldots,r_n\in\mathbb{Z}$ and $\alpha\in B_n$. The integers $r_1,\ldots,r_n$ are called the framings of the corresponding framed braid. With our convention for the action of $S_n$, conjugation is given by
$$
\alpha t_i\alpha^{-1}=t_{\pi(\alpha)^{-1}(i)}.
$$

Recall that the pure braid group is $P_n=\ker{\pi}$. It is generated by the Artin pure braids $A_{i,j}=\sigma_{j-1}\sigma_{j-2}\cdots\sigma_{i+1}\sigma_i^2\sigma_{i+1}^{-1}\cdots\sigma_{j-2}^{-1}\sigma_{j-1}^{-1}$, where $1\leq i<j\leq n$. As in \cite{GGO1}, whenever necessary we use the convention $A_{j,i}=A_{i,j}$. The abelianization of $P_n$ is $\faktor{P_n}{[P_n,P_n]}\cong\mathbb{Z}^{n(n-1)/2}$, with basis given by the images of the elements $A_{i,j}$, $1\leq i<j\leq n$.

The canonical homomorphism $\pi\colon B_n\longrightarrow S_n$ extends to an epimorphism
$$
\psi_0\colon FB_n\longrightarrow S_n
$$
by setting $\psi_0(t_i)=1$ for every $i=1,\ldots,n$. Its kernel is the framed pure braid group $FP_n$. Since $[FP_n,FP_n]\leq FP_n=\ker{\psi_0}$, the homomorphism $\psi_0$ induces an epimorphism
$$
\psi\colon\faktor{FB_n}{[FP_n,FP_n]}\longrightarrow S_n.
$$ 
Since $P_n$ acts trivially on $\mathbb{Z}^n$, we have
$$
FP_n=\mathbb{Z}^n\times P_n.
$$
Thus every element of $FP_n$ can be written uniquely as $t_1^{r_1}\cdots t_n^{r_n}\beta$, where $r_1,\ldots,r_n\in\mathbb{Z}$ and $\beta\in P_n$. The direct product decomposition also yields a presentation of $FP_n$, which we record next.

\begin{prop}\label{prop:FP-presentation}
The framed pure braid group $FP_n$ admits a presentation with generators $t_{\ell}$, for $1\leq \ell\leq n$, $A_{i,j}$, for $1\leq i<j\leq n$ and relations: 
\begin{itemize}
    \item  $A_{i,j}A_{r,s}=A_{r,s}A_{i,j}$ if  $r<s<i<j$  or $i<r<s<j$,
    \item $A_{r,s}^{-1}A_{i,j}A_{r,s}=A_{r,j}A_{i,j}A_{r,j}^{-1}$ if $r<s=i<j$,
    \item $A_{r,s}^{-1}A_{i,j}A_{r,s}=A_{r,j}A_{s,j}A_{i,j}A_{s,j}^{-1}A_{r,j}^{-1}$ if $i=r<s<j$,
    \item $A_{r,s}^{-1}A_{i,j}A_{r,s}=A_{r,j}A_{s,j}A_{r,j}^{-1}A_{s,j}^{-1}A_{i,j}A_{s,j}A_{r,j}A_{s,j}^{-1}A_{r,j}^{-1}$ if $r<i<s<j$,
    \item $t_{\ell}t_k=t_kt_{\ell}$ for all $1\leq \ell,k\leq n$,
    \item $t_{\ell}A_{i,j}=A_{i,j}t_{\ell}$ for all $1\leq \ell\leq n$, $1\leq i<j\leq n$.
\end{itemize}
\end{prop}

\begin{proof}
Since $FP_n=\mathbb Z^n\times P_n$, the result follows from the standard presentation of the pure braid group $P_n$ in terms of the generators $A_{i,j}$; see \cite{Artin}. Indeed, one only needs to adjoin the commuting generators $t_1,\ldots,t_n$ of $\mathbb Z^n$ together with the relations $[t_{\ell},A_{i,j}]=1$.
\end{proof}

The same direct product decomposition immediately yields the following consequence.

\begin{prop}\label{comFP}
Let $n\geq2$. Then $[FP_n,FP_n]=[P_n,P_n]$. 
\end{prop}

\begin{proof}
Since $FP_n=\mathbb{Z}^n\times P_n$ and $\mathbb{Z}^n$ is abelian, we have
$$
[FP_n,FP_n]=[\mathbb{Z}^n\times P_n,\mathbb{Z}^n\times P_n]=[P_n,P_n].
$$
\end{proof}

We shall also use the following elementary fact about quotients of semidirect products.

\begin{lema}\label{deco}
Let $G=H\rtimes K$, and let $N\trianglelefteq G$ be such that $N\leq K$ and $N$ acts trivially on $H$. Then the action of $K$ on $H$ induces an action of $\faktor{K}{N}$ on $H$, and $\faktor{G}{N}\cong H\rtimes(\faktor{K}{N})$. 
\end{lema}

\begin{proof}
Let $\varphi\colon \faktor{K}{N}\longrightarrow\operatorname{Aut}(H)$ be defined by $\varphi(kN)(h)=khk^{-1}$. Since $N$ acts trivially on $H$, the map $\varphi$ is well defined. Consider the homomorphism $\Phi\colon H\rtimes K\longrightarrow H\rtimes_\varphi(\faktor{K}{N})$ defined by $\Phi(h,k)=(h,kN)$. It is surjective and
$$
\ker{\Phi}=\{(1,n)\mid n\in N\}=N.
$$
Hence the first isomorphism theorem gives $\faktor{G}{N}\cong H\rtimes_\varphi(\faktor{K}{N})$. 
\end{proof}

Proposition~\ref{comFP} and Lemma~\ref{deco} yield a useful decomposition of the quotient of the framed braid group by the commutator subgroup of the framed pure braid group.

\begin{teo}\label{teodeco}
Let $n\geq2$. Then
$$
\faktor{FB_n}{[FP_n,FP_n]}\cong\mathbb{Z}^n\rtimes\faktor{B_n}{[P_n,P_n]},
$$
where the action of $\faktor{B_n}{[P_n,P_n]}$ on $\mathbb{Z}^n$ is induced by $\overline{\pi}\colon \faktor{B_n}{[P_n,P_n]}\longrightarrow S_n$.
\end{teo}

\begin{proof}
Since $FP_n$ is a normal subgroup of $FB_n$ and $[FP_n,FP_n]$ is characteristic in $FP_n$, it follows that $[FP_n,FP_n]\trianglelefteq FB_n$. By Proposition~\ref{comFP}, $[FP_n,FP_n]=[P_n,P_n]$. Moreover, $[P_n,P_n]\leq P_n$, and $P_n$ acts trivially on $\mathbb{Z}^n$. Applying Lemma~\ref{deco} to $FB_n=\mathbb{Z}^n\rtimes B_n$ with $N=[P_n,P_n]$, we obtain
$$
\faktor{FB_n}{[FP_n,FP_n]} = \faktor{FB_n}{[P_n,P_n]} \cong\mathbb{Z}^n\rtimes\faktor{B_n}{[P_n,P_n]}.
$$
The induced action factors through the canonical projection $\faktor{B_n}{[P_n,P_n]}\longrightarrow S_n$, as required.
\end{proof}

\subsection{The crystallographic structure}\label{subsec:crystallographic-structure}

We now describe explicitly the conjugation action of $\faktor{FB_n}{[FP_n,FP_n]}$ on $\faktor{FP_n}{[FP_n,FP_n]}$. Recall that
$$
\faktor{FP_n}{[FP_n,FP_n]}\cong\mathbb{Z}^n\times\faktor{P_n}{[P_n,P_n]}\cong\mathbb{Z}^{n(n+1)/2},
$$
with basis $\{t_1,\ldots,t_n\}\cup\{A_{i,j}\mid 1\leq i<j\leq n\}$. As before, we use the convention $A_{j,i}=A_{i,j}$.

\begin{prop}\label{prop:action}
Let $\alpha\in\faktor{FB_n}{[FP_n,FP_n]}$. By Theorem~\ref{teodeco}, write $\alpha=t_1^{a_1}\cdots t_n^{a_n}\beta$, where $\beta\in\faktor{B_n}{[P_n,P_n]}$, and let $\rho=\psi(\alpha^{-1})=\overline{\pi}(\beta^{-1})\in S_n$. Then 
\begin{align*}
\alpha A_{i,j}\alpha^{-1} & = A_{\rho(i),\rho(j)}\\    
\alpha t_k\alpha^{-1} & = t_{\rho(k)}
\end{align*}
for every $1\leq i<j\leq n$ and $1\leq k\leq n$.
\end{prop}

\begin{proof}
Write $\alpha=z\beta$, where $z=t_1^{a_1}\cdots t_n^{a_n}\in\mathbb{Z}^n$. Since $\faktor{FP_n}{[FP_n,FP_n]}\cong\mathbb{Z}^n\times\faktor{P_n}{[P_n,P_n]}$, the element $z$ commutes with every $A_{i,j}$. Hence 
$$
\alpha A_{i,j}\alpha^{-1}=z(\beta A_{i,j}\beta^{-1})z^{-1}=\beta A_{i,j}\beta^{-1}.
$$
The conjugation action of $\faktor{B_n}{[P_n,P_n]}$ on $\faktor{P_n}{[P_n,P_n]}$ is given by permutation of the indices; see \cite[Proposition~12]{GGO1}. Therefore,
$$
\beta A_{i,j}\beta^{-1}=A_{\rho(i),\rho(j)},
$$
where $\rho=\overline{\pi}(\beta^{-1})$.

Similarly, the action of $\faktor{B_n}{[P_n,P_n]}$ on the framing subgroup $\mathbb{Z}^n=\langle t_1,\ldots,t_n\rangle$ is induced by the natural permutation action on the indices. Thus
$$
\beta t_k\beta^{-1}=t_{\rho(k)}.
$$
Since $\mathbb{Z}^n$ is abelian, $z$ commutes with $t_{\rho(k)}$, and consequently
$$
\alpha t_k\alpha^{-1}=z(\beta t_k\beta^{-1})z^{-1}=zt_{\rho(k)}z^{-1}=t_{\rho(k)}.
$$
This proves the result.
\end{proof}

The preceding proposition shows that the conjugation action on the translation subgroup is the direct sum of two permutation representations: the natural action of $S_n$ on the framing generators $t_1,\ldots,t_n$ and the action of $S_n$ on the generators $A_{i,j}$ indexed by the two-element subsets of $\{1,\ldots,n\}$. We can now determine the crystallographic structure of the quotient.

\begin{teo}\label{teo:crist}
Let $n\geq3$. There exists a short exact sequence
$$
1\longrightarrow\mathbb{Z}^{n(n+1)/2}\longrightarrow\faktor{FB_n}{[FP_n,FP_n]}\stackrel{\psi}{\longrightarrow}S_n\longrightarrow1.
$$
Moreover, $\faktor{FB_n}{[FP_n,FP_n]}$ is a crystallographic group of dimension $\frac{n(n+1)}{2}$ with holonomy group $S_n$.
\end{teo}

\begin{proof}
Since $FP_n=\ker{\psi_0}$ and $[FP_n,FP_n]\trianglelefteq FB_n$, the epimorphism $\psi_0\colon FB_n\longrightarrow S_n$ induces the short exact sequence
$$
1\longrightarrow\faktor{FP_n}{[FP_n,FP_n]}\longrightarrow\faktor{FB_n}{[FP_n,FP_n]}\stackrel{\psi}{\longrightarrow}S_n\longrightarrow1.
$$
By Proposition~\ref{comFP}, $\faktor{FP_n}{[FP_n,FP_n]}\cong\mathbb{Z}^n\times\faktor{P_n}{[P_n,P_n]}$. Since $\faktor{P_n}{[P_n,P_n]}$ is free abelian of rank $\binom{n}{2}$, it follows that 
$$
\faktor{FP_n}{[FP_n,FP_n]}\cong\mathbb{Z}^{n+\binom{n}{2}}=\mathbb{Z}^{n(n+1)/2}.
$$

It remains to prove that the integral representation
$$
\Theta\colon S_n\longrightarrow\operatorname{Aut}\left(\faktor{FP_n}{[FP_n,FP_n]}\right)\cong\operatorname{GL}\left(\frac{n(n+1)}{2},\mathbb{Z}\right),
$$
induced by conjugation, is faithful. By Proposition~\ref{prop:action}, the free abelian subgroup generated by $t_1,\ldots,t_n$ is invariant under the action of $S_n$ induced by $\Theta$, and the resulting action on this subgroup is the natural permutation representation of $S_n$. More precisely, if $\rho\in S_n$, then
$$
\Theta(\rho)(t_k)=t_{\rho^{-1}(k)},\qquad 1\leq k\leq n.
$$
If $\Theta(\rho)$ is the identity, then $t_{\rho^{-1}(k)}=t_k$ for every $k$. Since $t_1,\ldots,t_n$ are distinct elements of a free basis, we obtain $\rho(k)=k$ for every $k$, and hence $\rho=1$. Thus $\Theta$ is faithful.

Therefore, Proposition~\ref{caraccrys} implies that $\faktor{FB_n}{[FP_n,FP_n]}$ is crystallographic of dimension $\frac{n(n+1)}{2}$ and has holonomy group $S_n$.
\end{proof}

\subsection{Torsion and conjugacy in $\faktor{FB_{n}}{[FP_{n},FP_{n}]}$}

We begin with a simple observation concerning finite-order elements in group extensions with torsion-free kernel.

\begin{lema}\label{lem:torsion-order}
Let $1\longrightarrow A \longrightarrow G \stackrel{\phi}{\longrightarrow} Q \longrightarrow 1$ be a short exact sequence of groups, where $A$ is torsion-free. If $g\in G$ has finite order, then $|g|=|\phi(g)|$. In particular, if $Q$ has no elements of finite even order, then neither does $G$.
\end{lema}

\begin{proof}
Let $g\in G$ have finite order $m$, and let $r=|\phi(g)|$. Since $\phi(g)^m=1$, we have $r\mid m$. Moreover, $\phi(g^r)=1$, so $g^r\in A$. Since $g$ has finite order, so does $g^r$. As $A$ is torsion-free, it follows that $g^r=1$, and hence $m\mid r$. Therefore $m=r$. The last assertion follows immediately.
\end{proof}

We apply this observation to the exact sequence obtained from the semidirect product decomposition in Theorem~\ref{teodeco}.

\begin{teo}\label{par}
Let $n\geq3$. Then $\faktor{FB_n}{[FP_n,FP_n]}$ has no elements of finite even order.
\end{teo}

\begin{proof}
By Theorem~\ref{teodeco}, there is a short exact sequence
$$
1\longrightarrow\mathbb Z^n\longrightarrow\faktor{FB_n}{[FP_n,FP_n]}\longrightarrow\faktor{B_n}{[P_n,P_n]}\longrightarrow1.
$$
The group $\mathbb Z^n$ is torsion-free, and $\faktor{B_n}{[P_n,P_n]}$ has no elements of finite even order by \cite[Theorem~2]{GGO1}. The result therefore follows from Lemma~\ref{lem:torsion-order}.
\end{proof}

The splitting of the exact sequence above also allows us to transfer finite-order elements from $\faktor{B_n}{[P_n,P_n]}$ to the framed quotient.

\begin{teo}\label{impar}
Let $n\geq3$ be odd. Then $\faktor{FB_n}{[FP_n,FP_n]}$ contains infinitely many elements of order $n$.
\end{teo}

\begin{proof}
By Theorem~\ref{teodeco}, $\faktor{FB_n}{[FP_n,FP_n]}\cong\mathbb Z^n\rtimes\faktor{B_n}{[P_n,P_n]}$. Hence, the short exact sequence
$$
1\longrightarrow\mathbb Z^n\longrightarrow\faktor{FB_n}{[FP_n,FP_n]}\stackrel{\xi}{\longrightarrow}\faktor{B_n}{[P_n,P_n]}\longrightarrow1
$$
splits. Let $s\colon\faktor{B_n}{[P_n,P_n]}\longrightarrow\faktor{FB_n}{[FP_n,FP_n]}$ be a section of $\xi$. Then $s$ is injective. Since $n$ is odd, $\faktor{B_n}{[P_n,P_n]}$ contains infinitely many elements of order $n$ by \cite[Proposition~21]{GGO1}. Their images under $s$ are therefore infinitely many elements of order $n$ in $\faktor{FB_n}{[FP_n,FP_n]}$.
\end{proof}

More generally, every element of odd order in $S_n$ admits a lift of the same order to the framed quotient.

\begin{coro}\label{cor:odd-lifting}
Let $n\geq3$, and let $\theta\in S_n$ have odd order $r$. Then there exists $\widetilde{\beta}\in\faktor{FB_n}{[FP_n,FP_n]}$ of order $r$ such that $\psi(\widetilde{\beta})=\theta$. Consequently, if $H\leq S_n$ is cyclic of odd order $r$, then there exists a cyclic subgroup $\widetilde{H}\leq\faktor{FB_n}{[FP_n,FP_n]}$ of order $r$ such that $\psi(\widetilde{H})=H$.
\end{coro}

\begin{proof}
Let $\overline{\pi}\colon\faktor{B_n}{[P_n,P_n]}\longrightarrow S_n$ be the natural projection. By \cite[Corollary~4]{GGO1}, there exists $\beta\in\faktor{B_n}{[P_n,P_n]}$ such that $|\beta|=r$ and $\overline{\pi}(\beta)=\theta$.

By Theorem~\ref{teodeco}, the exact sequence
$$
1\longrightarrow\mathbb Z^n\longrightarrow\faktor{FB_n}{[FP_n,FP_n]}\stackrel{\xi}{\longrightarrow}\faktor{B_n}{[P_n,P_n]}\longrightarrow1
$$
splits. Let $s\colon\faktor{B_n}{[P_n,P_n]}\longrightarrow\faktor{FB_n}{[FP_n,FP_n]}$ be a section, so that $\xi\circ s=\operatorname{Id}$. Set $\widetilde{\beta}=s(\beta)$. Since $s$ is injective, $|\widetilde{\beta}|=|\beta|=r$. Moreover, $\psi=\overline{\pi}\circ\xi$, and hence
$$
\psi(\widetilde{\beta})=\overline{\pi}(\xi(s(\beta)))=\overline{\pi}(\beta)=\theta.
$$

If $H=\langle\theta\rangle$, then $\widetilde{H}=\langle\widetilde{\beta}\rangle$ has order $r$ and
$$
\psi(\widetilde{H})=\langle\psi(\widetilde{\beta})\rangle=\langle\theta\rangle=H.
$$
\end{proof}

We now apply the general results on finite-order elements and conjugacy classes in free abelian-by-finite groups obtained in \cite[Section~2]{CO}. By Proposition~\ref{prop:action} and Theorem~\ref{teo:crist}, we have a short exact sequence
$$
1\longrightarrow\faktor{FP_n}{[FP_n,FP_n]}\longrightarrow\faktor{FB_n}{[FP_n,FP_n]}\stackrel{\psi}{\longrightarrow}S_n\longrightarrow1,
$$
where $\faktor{FP_n}{[FP_n,FP_n]}\cong\mathbb Z^{n(n+1)/2}$ is the free abelian group with basis
$$
X=\{t_1,\ldots,t_n\}\cup\{A_{i,j}\mid 1\leq i<j\leq n\}.
$$
Moreover, $S_n$ acts faithfully on this free abelian group by permuting the elements of $X$. Thus, the hypotheses of \cite[Theorems~2.1, 2.3 and 2.4]{CO} are satisfied.

For $\alpha\in\faktor{FB_n}{[FP_n,FP_n]}$, set $\theta=\psi(\alpha)^{-1}$. By Proposition~\ref{prop:action}, conjugation by $\alpha$ permutes the elements of $X$ according to $\theta$. We denote by $\mathbb O_\theta(\lambda)$ the orbit of $\lambda\in X$ under $\langle\theta\rangle$, and by $T_\theta\subseteq X$ a set containing exactly one representative of each such orbit.

The following result gives a criterion for finite order inside a fixed coset of $\faktor{FP_n}{[FP_n,FP_n]}$.

\begin{teo}\label{teo:ordem}
Let $\alpha\in\faktor{FB_n}{[FP_n,FP_n]}$ have finite order $t$, and let $\theta=\psi(\alpha)^{-1}$. For $a_i\in\mathbb Z$, the element $\left(\prod_{\lambda_i\in X}\lambda_i^{a_i}\right)\alpha$ has order $t$ if and only if $\sum_{\lambda_r\in O_\theta(\lambda_j)}a_r=0$ for every $\lambda_j\in T_\theta$.
\end{teo}

\begin{proof}
This is a direct application of \cite[Theorem~2.1]{CO}, taking $G=\faktor{FB_n}{[FP_n,FP_n]}$ and $\mathbb ZX=\faktor{FP_n}{[FP_n,FP_n]}$. The required permutation action of $S_n$ on $X$ is precisely the action described in Proposition~\ref{prop:action}. 
\end{proof}

The orbit sums appearing in Theorem~\ref{teo:ordem} also provide a normal form, up to conjugacy, for the elements lying over a fixed permutation.

\begin{teo}\label{teo:conj1}
Let $\alpha\in\faktor{FB_n}{[FP_n,FP_n]}$, let $\theta=\psi(\alpha)^{-1}$, and let $T_\theta$ be a transversal of the $\langle\theta\rangle$-orbits in $X$. For $g=\left(\prod_{\lambda_i\in X}\lambda_i^{a_i}\right)\alpha$, define $S_j=\sum_{\lambda_r\in O_\theta(\lambda_j)}a_r$ for every $\lambda_{j}\in T_\theta$. Then $g$ is conjugate in $\faktor{FB_n}{[FP_n,FP_n]}$ to $\left(\prod_{\lambda_j\in T_\theta}\lambda_j^{S_j}\right)\alpha$. 
\end{teo}

\begin{proof}
This follows directly from \cite[Theorem~2.3]{CO}, with $G=\faktor{FB_n}{[FP_n,FP_n]}$ and $\mathbb ZX=\faktor{FP_n}{[FP_n,FP_n]}$. 
\end{proof}

In general, conjugacy of the images of two elements in $S_n$ does not imply conjugacy of the elements themselves; see the discussion preceding \cite[Theorem~2.3]{CO}. For finite-order elements, however, the situation is particularly simple.

\begin{coro}\label{cor:torsion-conjugacy}
Let $\alpha_1,\alpha_2\in\faktor{FB_n}{[FP_n,FP_n]}$ be elements of finite order. Then $\alpha_1$ and $\alpha_2$ are conjugate in $\faktor{FB_n}{[FP_n,FP_n]}$ if and only if $\psi(\alpha_1)$ and $\psi(\alpha_2)$ are conjugate in $S_n$.
\end{coro}

\begin{proof}
If $\alpha_1$ and $\alpha_2$ are conjugate, then their images under $\psi$ are conjugate in $S_n$.

Conversely, suppose that $\psi(\alpha_1)$ and $\psi(\alpha_2)$ are conjugate in $S_n$. Since $\psi$ is surjective, after conjugating $\alpha_2$ in $\faktor{FB_n}{[FP_n,FP_n]}$ if necessary, we may assume that $\psi(\alpha_1)=\psi(\alpha_2)$. Hence there exists 
$$
z=\prod_{\lambda_i\in X}\lambda_i^{a_i}\in\faktor{FP_n}{[FP_n,FP_n]}
$$
such that $\alpha_2=z\alpha_1$. 

Set $\theta=\psi(\alpha_1)^{-1}$. By Lemma~\ref{lem:torsion-order},
$$
|\alpha_1|=|\psi(\alpha_1)|=|\psi(\alpha_2)|=|\alpha_2|.
$$
Since $\alpha_1$ has finite order, Theorem~\ref{teo:ordem}, applied to $\alpha_2=z\alpha_1$, implies that
$$
\sum_{\lambda_r\in\mathbb O_\theta(\lambda_{j})}a_r=0
$$
for every $\lambda_{j}\in T_\theta$. Therefore, Theorem~\ref{teo:conj1} shows that $\alpha_2=z\alpha_1$ is conjugate to
$$
\left(\prod_{\lambda_{j}\in T_\theta}\lambda_{j}^{0}\right)\alpha_1=\alpha_1.
$$
Hence $\alpha_1$ and $\alpha_2$ are conjugate.
\end{proof}

We conclude this subsection with the more general conjugacy criterion arising from \cite[Theorem~2.4]{CO}. Let $C_{S_n}(\theta)$ denote the centralizer of $\theta$ in $S_n$. For $c\in S_n$ and $\lambda_r\in X$, we write $\lambda_{c(r)}$ for the image of $\lambda_r$ under the permutation of $X$ induced by $c$.

\begin{teo}\label{teo:conj2}
Let $\alpha\in\faktor{FB_n}{[FP_n,FP_n]}$, let $\theta=\psi(\alpha)^{-1}$, and let $T_\theta$ be a transversal of the $\langle\theta\rangle$-orbits in $X$. Let $c\in C_{S_n}(\theta)$, and choose $\widetilde{c}\in\faktor{FB_n}{[FP_n,FP_n]}$ such that $\psi(\widetilde{c})^{-1}=c$ and $\widetilde{c}\alpha\widetilde{c}^{-1} = \left(\prod_{\lambda_i\in X}\lambda_i^{c_i}\right)\alpha$. Then there exists $Y\in\ker{\psi}$ such that
$$
Y\widetilde c \left(\prod_{\lambda_r\in T_\theta}\lambda_r^{a_r}\right) \alpha (Y\widetilde c)^{-1} = \left(\prod_{\lambda_r\in T_\theta}\lambda_r^{b_r}\right) \alpha
$$
if and only if, for every $\lambda_r\in T_\theta$, one of the following conditions holds: 
\begin{enumerate}
\item[(i)] if $\lambda_{c^{-1}(r)}\in O_\theta(\lambda_r)$, then $\sum_{\lambda_i\in O_\theta(\lambda_r)}c_i+a_r=b_r$; 

\item[(ii)] if $\lambda_{c^{-1}(r)}\notin O_\theta(\lambda_r)$, let $\lambda_s\in T_\theta$ be the unique representative such that $\lambda_{c^{-1}(r)}\in O_\theta(\lambda_s)$. Then $\sum_{\lambda_i\in O_\theta(\lambda_r)}c_i+a_s=b_r$. 
\end{enumerate}
\end{teo}

\begin{proof}
This follows from the proof of \cite[Theorem~2.4]{CO}, applied to $G=\faktor{FB_n}{[FP_n,FP_n]}$ and $\mathbb Z X=\faktor{FP_n}{[FP_n,FP_n]}$, using the faithful permutation action described in Proposition~\ref{prop:action}. Indeed, for the fixed element $c\in C_{S_n}(\theta)$ and the chosen lift $\widetilde c$, that proof amounts precisely to determining when there exists $Y\in \mathbb Z X=\ker{\psi}$ such that the displayed conjugacy relation holds.

In case~(ii), the notation in \cite[Theorem~2.4]{CO} involves the sum of the coefficients over the orbit containing $\lambda_{c^{-1}(r)}$. Since the normal form is supported on the transversal $T_\theta$, this sum is precisely $a_s$, where $\lambda_s\in T_\theta$ is the unique representative of that orbit.
\end{proof}

Theorems~\ref{teo:conj1} and~\ref{teo:conj2} provide, respectively, a normal form up to conjugacy for elements lying over a fixed permutation and, for each $c\in C_{S_n}(\theta)$, a criterion for conjugacy by an element of the form $Y\widetilde c$, with $Y\in\ker{\psi}$. Thus, by varying $c$ in $C_{S_n}(\theta)$, one obtains a criterion for deciding when two such normal forms are conjugate. An explicit procedure illustrating the use of these results is given in \cite[p.~577]{CO}.

\subsection{Bieberbach subgroups of $\faktor{FB_n}{[FP_n,FP_n]}$}

The crystallographic structure established above provides a natural way to construct Bieberbach subgroups. We shall show that the preimage of any $2$-subgroup of the holonomy group $S_n$ is torsion-free and hence is a Bieberbach group.

\begin{lema}\label{lem:bieber}
Let $n\geq3$ and let $H\leq S_n$ be a $2$-subgroup. Consider the natural projection $\psi\colon\faktor{FB_n}{[FP_n,FP_n]}\longrightarrow S_n$ and define $\widetilde{H}=\psi^{-1}(H)$. Then $\widetilde{H}$ is a Bieberbach group of dimension $n(n+1)/2$ with holonomy group $H$.
\end{lema}

\begin{proof}
By Theorem~\ref{teo:crist}, the group $\faktor{FB_n}{[FP_n,FP_n]}$ is crystallographic of dimension $n(n+1)/2$ with holonomy group $S_n$. Hence, by Corollary~\ref{holocor}, the subgroup $\widetilde{H}=\psi^{-1}(H)$ is crystallographic of dimension $n(n+1)/2$ with holonomy group $H$.

It remains to prove that $\widetilde{H}$ is torsion-free. Let $h\in\widetilde{H}$ have finite order. Since
$$
\ker{\psi}=\faktor{FP_n}{[FP_n,FP_n]}\cong\mathbb Z^{n(n+1)/2}
$$
is torsion-free, Lemma~\ref{lem:torsion-order} gives $|h|=|\psi(h)|$. Since $\psi(h)\in H$, its order is a power of $2$. On the other hand, Theorem~\ref{par} implies that $h$ cannot have nontrivial finite even order. Therefore $|h|=1$, and hence $h=1$.

Thus $\widetilde{H}$ is torsion-free and consequently is a Bieberbach group of dimension $n(n+1)/2$ with holonomy group $H$.
\end{proof}

Lemma~\ref{lem:bieber} realizes every $2$-subgroup of $S_n$ as the holonomy group of a Bieberbach subgroup of $\faktor{FB_n}{[FP_n,FP_n]}$. Combining this construction with Cayley's theorem yields the following realization result.

\begin{teo}\label{bieberbach}
Let $H$ be a finite $2$-group. For every embedding $H\hookrightarrow S_n$ with $n\geq3$, there exists a compact flat manifold $M$ of dimension $\frac{n(n+1)}{2}$ whose holonomy group is isomorphic to $H$ and whose fundamental group is isomorphic to a subgroup of $\faktor{FB_n}{[FP_n,FP_n]}$.
\end{teo}

\begin{proof}
Fix an embedding $H\hookrightarrow S_n$, with $n\geq3$, and identify $H$ with its image in $S_n$. Let
$$
\widetilde{H}=\psi^{-1}(H)\leq\faktor{FB_n}{[FP_n,FP_n]}.
$$
By Lemma~\ref{lem:bieber}, $\widetilde{H}$ is a Bieberbach group of dimension $n(n+1)/2$ with holonomy group $H$.

Therefore $\widetilde{H}$ acts freely, properly discontinuously, and cocompactly by isometries on $\mathbb R^{n(n+1)/2}$. Hence $M=\mathbb R^{n(n+1)/2}/\widetilde{H}$ is a compact flat manifold satisfying
$$
\pi_1(M)\cong\widetilde{H}\leq\faktor{FB_n}{[FP_n,FP_n]},
$$
and its holonomy group is isomorphic to $H$.
\end{proof}

By Cayley's theorem, every finite $2$-group admits an embedding into some symmetric group $S_n$ with $n\geq3$. Thus Theorem~\ref{bieberbach} applies to every finite $2$-group.

\section{Flat manifolds associated with the framed quotient}\label{sec:flat}

In this section, we use Bieberbach subgroups of $\faktor{B_n}{[P_n,P_n]}$ to construct Bieberbach subgroups of $\faktor{FB_n}{[FP_n,FP_n]}$. The semidirect product decomposition obtained in Theorem~\ref{teodeco} gives the resulting flat manifolds a natural geometric structure. Geometric fibrations arising from crystallographic quotients of Artin braid groups were studied in~\cite{Ocampo2}, using the theory of complete normal subgroups developed by Ratcliffe and Tschantz. In particular, Bieberbach subgroups of $\faktor{B_n}{[P_n,P_n]}$ were used there to construct torus bundles over compact flat $3$-manifolds. In the framed setting considered here, the additional permutation lattice $\mathbb Z^n$ leads naturally to torus bundles over flat manifolds arising from the corresponding unframed Bieberbach groups. We also study geometric and algebraic properties of these manifolds.

\subsection{Bieberbach subgroups and torus bundles}\label{subsec:torus-bundles}

Let $\Gamma$ be a Bieberbach group of dimension $d$, with translation subgroup $L\cong\mathbb Z^d$ and holonomy group $H\leq S_n$, fitting into a short exact sequence $1\longrightarrow L\longrightarrow\Gamma\xrightarrow{q}H\longrightarrow1$. Suppose that $\Gamma\leq\faktor{B_n}{[P_n,P_n]}$ and that the restriction to $\Gamma$ of the natural projection $\overline{\pi}\colon\faktor{B_n}{[P_n,P_n]}\longrightarrow S_n$ coincides with $q\colon\Gamma\longrightarrow H$, under the inclusion $H\leq S_n$.

If $\rho\colon H\longrightarrow GL(n,\mathbb Z)$ denotes the permutation representation induced by the action on the framing lattice $\mathbb Z^n$, so that $\rho(h)(e_i)=e_{h^{-1}(i)}$ for the standard basis $e_1,\ldots,e_n$, we may consider
$$
\widetilde{\Gamma}=\mathbb Z^n\rtimes_{\rho\circ q}\Gamma\leq\mathbb Z^n\rtimes\faktor{B_n}{[P_n,P_n]}\cong\faktor{FB_n}{[FP_n,FP_n]}.
$$

Let $\Theta\colon H\longrightarrow GL(d,\mathbb Z)$ be the holonomy representation of $\Gamma$ on $L$. Since $\Gamma$ is a Bieberbach group, $\Theta$ is faithful.

\begin{teo}\label{teo:framed-bieberbach}
With the notation above, $\widetilde{\Gamma}$ is a Bieberbach group of dimension $n+d$ with translation subgroup $\widetilde{L}=\mathbb Z^n\times L\cong\mathbb Z^{n+d}$ and holonomy group $H$. Its holonomy representation is equivalent to $\widetilde{\Theta}=\rho\oplus\Theta\colon H\longrightarrow GL(n+d,\mathbb Z)$.
\end{teo}

\begin{proof}
Since the action of $\Gamma$ on $\mathbb Z^n$ factors through $H$, the subgroup $L=\ker{q}$ acts trivially on $\mathbb Z^n$. Hence $\mathbb Z^n\rtimes L=\mathbb Z^n\times L$ is a free abelian normal subgroup of $\widetilde{\Gamma}$ of rank $n+d$. Moreover,
$$
\faktor{\widetilde{\Gamma}}{\mathbb Z^n\times L}\cong\faktor{\Gamma}{L}\cong H.
$$
The conjugation action of $H$ on $\mathbb Z^n\times L$ is the direct sum of the permutation representation $\rho$ on $\mathbb Z^n$ and the holonomy representation $\Theta$ on $L$. Thus the associated integral representation is $\widetilde{\Theta}=\rho\oplus\Theta$. Since $\Theta$ is faithful, $\widetilde{\Theta}$ is faithful. Proposition~\ref{caraccrys} therefore implies that $\widetilde{\Gamma}$ is crystallographic of dimension $n+d$ with holonomy group $H$.

It remains to prove that $\widetilde{\Gamma}$ is torsion-free. Let $(z,\gamma)\in\widetilde{\Gamma}$ have finite order. Its image $\gamma$ under the natural projection $\widetilde{\Gamma}\longrightarrow\Gamma$ has finite order. Since $\Gamma$ is torsion-free, $\gamma=1$. Hence $(z,\gamma)=(z,1)\in\mathbb Z^n$, and since $\mathbb Z^n$ is torsion-free, $z=0$. Therefore $\widetilde{\Gamma}$ is torsion-free and hence is a Bieberbach group.
\end{proof}

Fix a realization of $\Gamma$ as a discrete, cocompact, torsion-free subgroup of $\operatorname{Isom}(\mathbb R^d)$. Thus, if $\gamma\in\Gamma$ and $h=q(\gamma)$, we may write $\gamma(y)=a_\gamma+\Theta(h)y$. We realize $\widetilde{\Gamma}$ as a group of isometries of $\mathbb R^n\oplus\mathbb R^d$ by letting $(z,\gamma)\in\widetilde{\Gamma}$ act as
$$
(z,\gamma)(x,y)=\bigl(z+\rho(h)x,a_\gamma+\Theta(h)y\bigr).
$$
Let $M_\Gamma=\mathbb R^d/\Gamma$ and $M_{\widetilde{\Gamma}}=(\mathbb R^n\oplus\mathbb R^d)/\widetilde{\Gamma}$ be the corresponding compact flat manifolds. Let $T^n=\mathbb R^n/\mathbb Z^n$ denote the $n$-dimensional torus. 

\begin{teo}\label{teo:torus-bundle}
There is a geometric fiber bundle $T^n\longrightarrow M_{\widetilde{\Gamma}}\longrightarrow M_\Gamma$ with totally geodesic flat torus fibers. Moreover, this bundle admits an affine section.
\end{teo}

\begin{proof}
Let $N=\mathbb Z^n\times\{1\}\trianglelefteq\widetilde{\Gamma}$. Then $\faktor{\widetilde{\Gamma}}{N}\cong\Gamma$, which is torsion-free. By \cite[Theorem~13]{RT1}, $N$ is a complete normal subgroup of $\widetilde{\Gamma}$ and the quotient map induces a geometric fiber bundle
$$
V/N\longrightarrow M_{\widetilde{\Gamma}}\longrightarrow (\mathbb R^{n+d}/V)/\left(\faktor{\widetilde{\Gamma}}{N}\right),
$$
where $V=\operatorname{Span}_{\mathbb R}(N)=\mathbb R^n\oplus\{0\}$. Under the realization fixed above, $\mathbb R^{n+d}/V$ is naturally identified with $\mathbb R^d$, and the induced action of $\faktor{\widetilde{\Gamma}}{N}\cong\Gamma$ is precisely the given action of $\Gamma$ on $\mathbb R^d$. Hence
$$
(\mathbb R^{n+d}/V)/\left(\faktor{\widetilde{\Gamma}}{N}\right)\cong\mathbb R^d/\Gamma=M_\Gamma.
$$
Since $N\cong\mathbb Z^n$, we have $V/N\cong\mathbb R^n/\mathbb Z^n=T^n$. Thus $T^n\longrightarrow M_{\widetilde{\Gamma}}\longrightarrow M_\Gamma$ is a geometric fiber bundle with totally geodesic flat torus fibers.

The short exact sequence $1\longrightarrow\mathbb Z^n\longrightarrow\widetilde{\Gamma}\longrightarrow\Gamma\longrightarrow1$ splits, since $\widetilde{\Gamma}=\mathbb Z^n\rtimes\Gamma$. Therefore, by \cite[Lemma~6]{RT1}, the corresponding geometric fibration admits an affine section.
\end{proof}

Recall that, since $L$ is the translation subgroup of $\Gamma$, the natural projection
$$
T^d=\mathbb R^d/L\longrightarrow M_\Gamma=\mathbb R^d/\Gamma
$$
is a regular covering with deck transformation group $\Gamma/L\cong H$, called the \emph{holonomy covering} of $M_\Gamma$. The previous bundle has a particularly explicit description in terms of the generalized Calabi construction  of Ratcliffe and Tschantz~\cite{RT2}.

\begin{teo}\label{teo:calabi}
With the notation above, there is an isometry of compact flat manifolds
$$
M_{\widetilde{\Gamma}}\cong\faktor{T^n\times T^d}{H},
$$
where $T^d=\mathbb R^d/L$ is the holonomy covering of $M_\Gamma$, and $H$ acts diagonally on $T^n\times T^d$: its action on $T^n$ is induced by the permutation representation $\rho$, while its action on $T^d$ is the deck action induced by the identification $\Gamma/L\cong H$, so that $M_\Gamma\cong T^d/H$. Under this identification, the bundle of Theorem~\ref{teo:torus-bundle} is induced by the projection $T^n\times T^d\longrightarrow T^d$.

In particular, the pullback of this bundle to the holonomy covering $T^d\longrightarrow M_\Gamma$ is the trivial bundle $T^n\times T^d\longrightarrow T^d$.
\end{teo}

\begin{proof}
Let $V=\mathbb R^n\oplus\{0\}$ and $V^\perp=\{0\}\oplus\mathbb R^d$. As in the proof of Theorem~\ref{teo:torus-bundle}, the subgroup $N=\mathbb Z^n\times\{1\}$ is complete.

Let $K$ be the kernel associated with the complete normal subgroup $N$ in the sense of \cite[Theorem~3]{RT2}. Recall that an isometry $b+B\in\widetilde{\Gamma}$ belongs to $K$ if and only if
$$
b\in V^\perp \qquad\text{and}\qquad V\subseteq\operatorname{Fix}(B).
$$
For $(z,\gamma)\in\widetilde{\Gamma}$, with $h=q(\gamma)$, the translational and linear parts with respect to the decomposition $\mathbb R^n\oplus\mathbb R^d$ are $(z,a_\gamma)$ and $\rho(h)\oplus\Theta(h)$, respectively. Since $V=\mathbb R^n\oplus\{0\}$ and $V^\perp=\{0\}\oplus\mathbb R^d$, the condition $(z,a_\gamma)\in V^\perp$ is equivalent to $z=0$, while
$$
V\subseteq\operatorname{Fix}\bigl(\rho(h)\oplus\Theta(h)\bigr)
$$
is equivalent to $\rho(h)=I_n$. Since $\rho$ is faithful, this is equivalent to $h=1$, or $\gamma\in L$. Therefore $K=\{0\}\times L$.

Consequently, $NK=\mathbb Z^n\times L=\widetilde{L}$ and $\faktor{\widetilde{\Gamma}}{NK}\cong H$. By the generalized Calabi construction \cite[Theorem~4]{RT2},
$$
M_{\widetilde{\Gamma}}\cong\faktor{(V/N)\times(V^\perp/K)}{\left( \faktor{\widetilde{\Gamma}}{NK}\right)}.
$$
Since $V/N\cong T^n$, $V^\perp/K\cong T^d$ and $\faktor{\widetilde{\Gamma}}{NK}\cong H$, we obtain $M_{\widetilde{\Gamma}}\cong\faktor{T^n\times T^d}{H}$. The action on the first factor is induced by $\rho$, while the action on the second factor is the induced affine action for which $M_\Gamma\cong\faktor{T^d}{H}$, and the projection onto the second factor induces the torus bundle of Theorem~\ref{teo:torus-bundle}.

Since the bundle projection is induced by the $H$-equivariant projection $T^n\times T^d\longrightarrow T^d$, its pullback along the holonomy covering $T^d\longrightarrow T^d/H\cong M_\Gamma$ is naturally isomorphic to the product bundle $T^n\times T^d\longrightarrow T^d$.
\end{proof}

\subsection{Orientability, the center and the first Betti number}\label{subsec:invariants}

The decomposition $\widetilde{\Theta}=\rho\oplus\Theta$ allows several invariants of $M_{\widetilde{\Gamma}}$ to be expressed directly in terms of the corresponding invariants of $M_\Gamma$ and the permutation action of $H$ on $\{1,\ldots,n\}$. We denote  the sign homomorphism by $\operatorname{sgn}\colon S_n\longrightarrow\{\pm1\}$, and recall that $A_n=\ker{\operatorname{sgn}}$ is the alternating group.

\begin{prop}\label{prop:orientability}
The flat manifold $M_{\widetilde{\Gamma}}$ is orientable if and only if $\det\Theta(h)=\operatorname{sgn}(h)$ for every $h\in H$.
\end{prop}

\begin{proof}
A flat manifold is orientable if and only if its holonomy representation has determinant $1$ for every element of the holonomy group; see, for instance, \cite[Remark~5]{Ocampo}. Since $\widetilde{\Theta}(h)=\rho(h)\oplus\Theta(h)$, we have
$$
\det\widetilde{\Theta}(h)=\det\rho(h)\det\Theta(h).
$$
The determinant of the permutation representation is $\det\rho(h)=\operatorname{sgn}(h)$. Therefore $\det\widetilde{\Theta}(h)=1$ for every $h\in H$ if and only if $\det\Theta(h)=\operatorname{sgn}(h)$ for every $h\in H$.
\end{proof}


\begin{coro}\label{cor:orientable-base}
If $M_\Gamma$ is orientable, then $M_{\widetilde{\Gamma}}$ is orientable if and only if $H\leq A_n$.
\end{coro}

\begin{proof}
Since $M_\Gamma$ is orientable, $\det\Theta(h)=1$ for every $h\in H$. Proposition~\ref{prop:orientability} therefore implies that $M_{\widetilde{\Gamma}}$ is orientable if and only if $\operatorname{sgn}(h)=1$ for every $h\in H$, which is equivalent to $H\leq A_n$.
\end{proof}

We now determine the center and the first Betti number. For a compact manifold $M$, we denote by $b_1(M)$ its first Betti number,
$$
b_1(M)=\operatorname{rank}H_1(M;\mathbb Z).
$$  
Let $O_1,\ldots,O_s$ be the orbits of the action of $H$ on $\{1,\ldots,n\}$, and for each $1\leq j\leq s$ set
$$
t_{O_j}=\prod_{i\in O_j}t_i.
$$

\begin{prop}\label{prop:center-betti}
With the notation above, $Z(\widetilde{\Gamma})\cong Z(\Gamma)\times\mathbb Z^s$. More precisely, $Z(\widetilde{\Gamma})=(\mathbb Z^n)^H\times Z(\Gamma)$, and $(\mathbb Z^n)^H$ is the free abelian group with basis $t_{O_1},\ldots,t_{O_s}$. Consequently, $\operatorname{rank}Z(\widetilde{\Gamma})=\operatorname{rank}Z(\Gamma)+s$ and $b_1(M_{\widetilde{\Gamma}})=b_1(M_\Gamma)+s$. 
\end{prop}

\begin{proof}
For a Bieberbach group $\Pi$ with translation subgroup $\Lambda$ and holonomy group $F$, the center is the fixed subgroup
$$
Z(\Pi)=\Lambda^F=\{\lambda\in\Lambda\mid f\lambda=\lambda\text{ for every }f\in F\};
$$
see \cite[Lemma~5.2(3)]{Szczepanski}. Applying this to $\widetilde{\Gamma}$ and using $\widetilde{L}=\mathbb Z^n\times L$ together with $\widetilde{\Theta}=\rho\oplus\Theta$, we obtain
$$
Z(\widetilde{\Gamma})=\widetilde{L}^{H}=(\mathbb Z^n)^H\times L^H=(\mathbb Z^n)^H\times Z(\Gamma).
$$
Since $H$ acts on the basis $t_1,\ldots,t_n$ by permutation, an element $t_1^{a_1}\cdots t_n^{a_n}$ is fixed by $H$ if and only if the exponents are constant on each $H$-orbit. Hence $(\mathbb Z^n)^H$ is freely generated by $t_{O_1},\ldots,t_{O_s}$, and therefore has rank $s$.

For a compact flat manifold with Bieberbach fundamental group $\Pi$, translation subgroup $\Lambda$ and holonomy group $F$, one has $b_1(M_\Pi)=\dim_{\mathbb Q}(\Lambda\otimes\mathbb Q)^F$; this follows from \cite[Theorems~6.4.12 and~6.4.13]{Dekimpe}, since over $\mathbb Q$ the spaces of invariants and coinvariants of the $F$-module $\Lambda\otimes\mathbb Q$ have the same dimension (see also \cite[p.~118]{Szczepanski}).  Together with $Z(\Pi)=\Lambda^F$, this gives $b_1(M_\Pi) = \operatorname{rank}Z(\Pi)$. Applying this equality to $\Gamma$ and $\widetilde{\Gamma}$ yields
$$
b_1(M_{\widetilde{\Gamma}})=\operatorname{rank}Z(\widetilde{\Gamma})=\operatorname{rank}Z(\Gamma)+s=b_1(M_\Gamma)+s.
$$
\end{proof}

\subsection{Flat manifolds with cyclic holonomy}\label{subsec:odd-cyclic}

We now apply the preceding construction to the Bieberbach groups with finite abelian holonomy constructed in \cite{Ocampo}. We concentrate on the cyclic case.

Let $q=p_1^{r_1}\cdots p_t^{r_t}$ be an odd integer, where $p_1,\ldots,p_t$ are distinct odd primes and $r_i\geq1$ for every $1\leq i\leq t$, and put $q_i=p_i^{r_i}$ and $n=q_1+\cdots+q_t$. Let $\theta\in S_n$ be a product of $t$ disjoint cycles of lengths $q_1,\ldots,q_t$. Since the integers $q_1,\ldots,q_t$ are pairwise coprime, we have
$$
|\theta|=\operatorname{lcm}(q_1,\ldots,q_t)=q,
$$
and hence $\langle\theta\rangle\cong C_q$. From now on, we identify $\langle\theta\rangle$ with $C_q$, so that $\theta$ is regarded as a fixed generator of $C_q$. For a square matrix $A$, let $\chi_A(x)=\det(xI-A)$ denote its characteristic polynomial.

In \cite[Theorem~1]{Ocampo}, a Bieberbach subgroup $\Gamma_q\leq\faktor{B_n}{[P_n,P_n]}$ of dimension $\frac{n(n-1)}{2}$ and holonomy group $C_q$ was constructed, where the holonomy epimorphism is induced by the restriction of the natural projection $\overline{\pi}\colon\faktor{B_n}{[P_n,P_n]}\longrightarrow S_n$. Its holonomy representation 
$$
\psi_q\colon C_q\longrightarrow GL\left(\frac{n(n-1)}{2},\,\mathbb Z\right)
$$ 
satisfies, by \cite[Theorem~11]{Ocampo},
$$
\chi_{\psi_q(\theta)}(x)=\prod_{i=1}^t(x^{q_i}-1)^{(q_i-1)/2}\prod_{1\leq i<j\leq t}(x^{q_iq_j}-1).
$$

Define
$$
\widetilde{\Gamma}_q=\mathbb Z^n\rtimes\Gamma_q\leq\faktor{FB_n}{[FP_n,FP_n]}.
$$
By Theorem~\ref{teo:framed-bieberbach}, $\widetilde{\Gamma}_q$ is a Bieberbach group of dimension $\frac{n(n+1)}{2}$ and holonomy group $C_q$. Let $M_{\widetilde{\Gamma}_q}$ denote the corresponding compact flat manifold. We identify the holonomy representation of $\widetilde{\Gamma}_q$ with
$$
\widetilde{\psi}_q=\rho\oplus\psi_q\colon C_q\longrightarrow GL\left(\frac{n(n+1)}{2},\mathbb Z\right),
$$
where $\rho$ is the permutation representation of $C_q$ on $\mathbb Z^n$. 

\begin{teo}\label{teo:odd-cyclic}
The representation $\widetilde{\psi}_q$ satisfies
$$
\chi_{\widetilde{\psi}_q(\theta)}(x)=\prod_{i=1}^t(x^{q_i}-1)^{(q_i+1)/2}\prod_{1\leq i<j\leq t}(x^{q_iq_j}-1).
$$
Moreover, $b_1(M_{\widetilde{\Gamma}_q})=\operatorname{rank}Z(\widetilde{\Gamma}_q)=\frac{n+t^2}{2}$, and $M_{\widetilde{\Gamma}_q}$ is orientable. 
\end{teo}

\begin{proof}
The permutation representation $\rho$ of $C_q$ on $\mathbb Z^n$ is the direct sum of the permutation representations associated with the $t$ disjoint cycles of lengths $q_1,\ldots,q_t$. Since the characteristic polynomial of the permutation matrix associated with a cycle of length $q_i$ is $x^{q_i}-1$, we have
$$
\chi_{\rho(\theta)}(x)=\prod_{i=1}^t(x^{q_i}-1).
$$ 
Since $\widetilde{\psi}_q=\rho\oplus\psi_q$, we obtain
$$
\chi_{\widetilde{\psi}_q(\theta)}(x)=\chi_{\rho(\theta)}(x)\chi_{\psi_q(\theta)}(x)=\prod_{i=1}^t(x^{q_i}-1)^{(q_i+1)/2}\prod_{1\leq i<j\leq t}(x^{q_iq_j}-1).
$$

The action of $C_q$ on $\{1,\ldots,n\}$ has exactly $t$ orbits. By Proposition~\ref{prop:center-betti}, both the first Betti number and the rank of the center increase by $t$ under the framed construction. By \cite[Theorem~2(b) and Theorem~13]{Ocampo}, $b_1(M_{\Gamma_q}) = \operatorname{rank}Z(\Gamma_q) = \sum_{i=1}^t\frac{q_i-1}{2}+\frac{t(t-1)}{2}$. Therefore 
$$
b_1(M_{\widetilde{\Gamma}_q})=\operatorname{rank}Z(\widetilde{\Gamma}_q)=\sum_{i=1}^t\frac{q_i-1}{2}+\frac{t(t+1)}{2}=\frac{n+t^2}{2}.
$$

Finally, each $q_i$ is odd, so every cycle of $\theta$ has odd length. Thus $\theta$ is an even permutation and consequently $C_q\leq A_n$. The manifolds $M_{\Gamma_q}$ are orientable by \cite[Theorem~2(a)]{Ocampo}, and Corollary~\ref{cor:orientable-base} implies that $M_{\widetilde{\Gamma}_q}$ is orientable.
\end{proof}

The addition of the framing permutation module also has consequences for the existence of flat Kähler structures.

\begin{teo}\label{teo:kahler-odd}
The flat manifold $M_{\widetilde{\Gamma}_q}$ admits a flat K\"ahler structure if and only if $n+t^2\equiv0\pmod4$. Equivalently,
$$
M_{\widetilde{\Gamma}_q}\text{ admits a flat K\"ahler structure}\quad\Longleftrightarrow\quad
\begin{cases}
n\equiv0\pmod4,&\text{if }t\text{ is even},\\
n\equiv3\pmod4,&\text{if }t\text{ is odd}.
\end{cases}
$$
\end{teo}

\begin{proof}
By \cite[Theorem~1.1 and Proposition~1.2]{DHS}, a flat manifold admits a flat K\"ahler structure if and only if its real holonomy representation is essentially complex. Equivalently, its dimension is even and every irreducible real constituent whose complexification remains irreducible occurs with even multiplicity. Let
$$
\widetilde{\psi}_{q,\mathbb R}\colon C_q\longrightarrow GL\left(\frac{n(n+1)}{2},\mathbb R\right)
$$
denote the real holonomy representation obtained from $\widetilde{\psi}_q$ by extension of scalars.

Since $C_q$ is cyclic of odd order, the only irreducible real representation of $C_q$ whose complexification remains irreducible is the trivial representation. Indeed, every nontrivial irreducible real representation of $C_q$ is two-dimensional and arises from a pair of nonreal complex conjugate characters; its complexification is therefore reducible. 

Consequently, the real holonomy representation $\widetilde{\psi}_{q,\mathbb R}$ is essentially complex if and only if the multiplicity of the trivial representation is even. Notice that this condition also implies that the total dimension is even, since all nontrivial irreducible real constituents are two-dimensional.

The multiplicity of the trivial representation in $\widetilde{\psi}_{q,\mathbb R}$ is the dimension of its fixed-point subspace, which equals $b_1(M_{\widetilde{\Gamma}_q})$ by Proposition~\ref{prop:center-betti}. Hence, by Theorem~\ref{teo:odd-cyclic},
$$
b_1(M_{\widetilde{\Gamma}_q}) = \frac{n+t^2}{2}.
$$
Therefore $M_{\widetilde{\Gamma}_q}$ admits a flat K\"ahler structure if and only if $\frac{n+t^2}{2}\equiv0\pmod2$, or equivalently $n+t^2\equiv0\pmod4$. If $t$ is even, then $t^2\equiv0\pmod4$, and the condition becomes $n\equiv0\pmod4$. If $t$ is odd, then $t^2\equiv1\pmod4$, and the condition becomes $n\equiv3\pmod4$.
\end{proof}

In the particular case in which $q=p^r$ is an odd prime power, we have $t=1$ and $n=q$. Thus the preceding results take a particularly simple form.

\begin{coro}\label{cor:prime-power}
Let $q=p^r$, where $p$ is an odd prime and $r\geq1$. Then $\widetilde{\Gamma}_q$ is a Bieberbach group of dimension $\frac{q(q+1)}{2}$ with holonomy group $C_q$, and $b_1(M_{\widetilde{\Gamma}_q}) = \operatorname{rank}Z(\widetilde{\Gamma}_q)=\frac{q+1}{2}$. Moreover, $M_{\widetilde{\Gamma}_q}$ is orientable and admits a flat Kähler structure if and only if $q\equiv3\pmod 4$. Finally, there is a geometric fiber bundle $T^q\longrightarrow M_{\widetilde{\Gamma}_q}\longrightarrow M_{\Gamma_q}$ with totally geodesic flat torus fibers and an affine section, and
$$
M_{\widetilde{\Gamma}_q}\cong \left(T^q\times T^{q(q-1)/2}\right)/C_q.
$$
\end{coro}

We finally note that the same framed construction applies to the Bieberbach groups with cyclic holonomy of $2$-power order studied in \cite{OR}. Let $n\geq3$ and $\ell,m\geq1$, set $D=2^\ell$ and $k=Dm\leq n$, and let $H_{D,k}=\langle\mu\rangle\leq S_n$, where $\mu$ is the product of $m$ disjoint $D$-cycles and fixes the remaining $n-k$ points. Since $\mu$ is a product of disjoint $D$-cycles, it has order $D$, and hence $H_{D,k}\cong C_D$. Let
$$
\Gamma_{D,k}=\overline{\pi}^{-1}(H_{D,k})\leq\faktor{B_n}{[P_n,P_n]}.
$$
By \cite[Lemma~18]{GGO1}, $\Gamma_{D,k}$ is a Bieberbach group of dimension $\frac{n(n-1)}{2}$ with holonomy group $H_{D,k}$. This family, including its holonomy representation and geometric properties,  is studied in detail in \cite{OR}. By Theorem~\ref{teo:framed-bieberbach}, 
$$
\widetilde{\Gamma}_{D,k}=\mathbb Z^n\rtimes\Gamma_{D,k}\leq\faktor{FB_n}{[FP_n,FP_n]}
$$
is a Bieberbach group of dimension $\frac{n(n+1)}{2}$ with holonomy group $H_{D,k}$. Moreover, by Theorem~\ref{teo:torus-bundle}, there is a geometric fiber bundle
$$
T^n\longrightarrow M_{\widetilde{\Gamma}_{D,k}}\longrightarrow M_{\Gamma_{D,k}}
$$
with totally geodesic flat torus fibers and an affine section. The $m$ disjoint $D$-cycles determine $m$ orbits of cardinality $D$, while the remaining $n-k$ fixed points determine $n-k$ singleton orbits. Hence the action of $H_{D,k}$ on $\{1,\ldots,n\}$ has $s=m+n-k$ orbits. Therefore, Proposition~\ref{prop:center-betti} gives
$$
b_1(M_{\widetilde{\Gamma}_{D,k}}) = b_1(M_{\Gamma_{D,k}})+m+n-k 
$$
and
$$
\operatorname{rank}Z(\widetilde{\Gamma}_{D,k}) = \operatorname{rank}Z(\Gamma_{D,k})+m+n-k.
$$

\subsection*{Conflict of Interest Statement}
All authors have no conflicts of interest.

\bibliographystyle{alpha}

\end{document}